\documentclass[a4paper,12pt,twoside]{article}
\usepackage{amsmath,amsthm,amscd,amsfonts,amssymb,amsxtra}
\usepackage[hmargin=1.2in,vmargin=1.2in]{geometry}
\usepackage{graphicx}
\usepackage{array,longtable}
\usepackage[pdftex,bookmarks,colorlinks=false]{hyperref}
\usepackage{booktabs,enumitem,caption,subcaption}
\usepackage[mathscr]{euscript}
\usepackage{longtable}

\usepackage[auth-sc-lg,affil-it]{authblk}
\numberwithin{equation}{section}

\allowdisplaybreaks[4]

\newtheorem{thm}{Theorem}[section]
\newtheorem{defn}{Definition}[section]

\newtheorem{prop}[thm]{Proposition}
\newtheorem{cor}[thm]{Corollary}

\newtheorem{conj}{Conjecture}

\newtheorem{rem}{Remark}

\def\ni{\noindent}
\def\N{\mathbb{N}}
\def\J{\mathbb{J}}

\title{\textbf{\sc On the Vertex In-Degrees of Certain Jaco-Type Graphs}}

\author{Johan Kok}
\affil{\small Tshwane Metropolitan Police Department\\ City of Tshwane, South Africa \\ {\tt kokkiek2@tshwane.gov.za.}}

\author{N. K. Sudev}
\affil{\small Centre for Studies in Discrete Mathematics\\ Vidya Academy of Science \&   Technology \\ Thalakkottukara, Thrissur, India.\\ {\tt sudevnk@gmail.com}}

\author{K. P. Chithra}
\affil{\small Naduvath Mana, Nandikkara \\ Thrissur, India.\\ {\tt chithrasudev@gmail.com}}

\author{K. A. Germina}
\affil{\small Department of Mathematics,\\ University of Botswana \\ Gaborone, Botswana.\\ {\tt srgerminaka@gmail.com}}

\author{U. Mary}
\affil{\small Department of Mathematics\\ Nirmala College For Women\\  Coimbatore, India. \\{\tt marycbe@gmail.com}}

\date{}

\begin{document}
\maketitle
	
\vspace{-0.5cm}

\newpage

\begin{abstract}
\ni The concepts of linear Jaco graphs and Jaco-type graphs have been introduced as certain types of directed graphs with specifically defined adjacency conditions. The distinct difference between a pure Jaco graph and a Jaco-type graph is that for a pure Jaco graph, the total vertex degree $d(v)$ is well-defined, while for a Jaco-type graph the vertex out-degree $d^+(v)$ is well-defined. Hence, in the case of pure Jaco graphs a challenge is to determine $d^-(v)$ and $d^+(v)$ respectively and for Jaco-type graphs a challenge is to determine $d^-(v)$. In this paper, the vertex in-degrees for Fibonaccian and modular Jaco-type graphs are determined.
\end{abstract}
\ni \textbf{Keywords:} Jaco-type graph, Fibonaccian Jaco-type graph, modular Jaco-type graph, vertex in-degree.

\vspace{0.20cm}

\ni \textbf{Mathematics Subject Classification:} 05C07, 05C38, 05C75, 05C85.

\section{Introduction}

For general notation and concepts in graphs and digraphs see \cite{BM1,CL1,FH1,DBW}. Unless mentioned otherwise, all graphs in this paper are simple, connected and directed graphs (digraphs).
 
The concept of a special class of directed graphs, namely Jaco graphs, with a specific adjacency conditions was introduced in \cite{KFW1,KFW2}. The notion of Jaco graphs has been improved later and hence the notion of linear Jaco graphs, has been introduced as follows.

\begin{defn}{\rm 
\cite{KSS1} An infinite linear Jaco graph, denoted by $J_{\infty}(f(x))$, with $f(x) = mx + c,$ $x,m \in \N,$ $c \in \N_0$, is a directed graph with vertex set $\{v_i: i\in \N\}$ such that $(v_i,v_j)$ is an arc of $J_{\infty}(f(x))$ if and only if $f(i)+i-d^-(v_j)\ge j$.}
\end{defn}

A Jaco graph is considered to be a \textit{pure Jaco graph} if the vertex degree $d(v)$ is well-defined. The above mentioned studies are the main initial studies on the families of \textit{pure Jaco graphs}.  Further research followed on different classes of Jaco graphs in \cite{KSS1,JK1,KSC1} and a few more papers on different properties and characteristics of Jaco graphs followed subsequently. 

In \cite{KSC1}, it is reported that a linear Jaco graph $J_n(x)$ can be defined a the graphical embodiment of a specific sequence defining the vertex out-degree. This observation opened the scope for determining the graphical embodiment of countless other integer sequences and and studying their characteristics. 

These graphs (graphical embodiments) corresponding to different integer sequences, with well-defined vertex out-degrees are broadly named as \textit{Jaco-type graphs}. A general definition of a Jaco-type graph is as follows.

\begin{defn}\label{Defn-1.1}{\rm 
\cite{KSC1} For a non-negative, non-decreasing integer sequence $\{a_n\}$, an \textit{infinite Jaco-type graph}, denoted by $J_\infty(\{a_n\})$, is defined as a directed graph with vertex set $V(J_\infty(\{a_n\}))= \{v_i: i \in \N\}$ and the arc set $A(J_\infty(\{a_n\})) \subseteq \{(v_i, v_j): i, j \in \N, i< j\}$ such that $(v_i,v_ j) \in A(J_\infty(\{a_n\}))$ if and only if $a_i+ i \ge j$.
}\end{defn}

\begin{defn}\label{Defn-1.2}{\rm 
\cite{KSC1} For a non-negative, non-decreasing integer sequence $\{a_n\}$, the a \textit{finite Jaco-type Graph} denoted by $J_n(\{a_n\})$, is a finite subgraph of the infinite Jaco-type graph $J_\infty(\{a_n\});n\in \N$.
}\end{defn}

So far, the introductory research on Jaco-type graphs dealt with non-negative, non-decreasing integer sequences only. 

Note that a finite Jaco-type graph $J_n(\{a_n\})$ is obtained from $J_\infty(\{a_n\})$ by lobbing off all vertices $v_k$ (with incident arcs) for all $k>n$.

Note that, the total vertex degree $d(v)$ of each vertex $v$ of a pure Jaco graph is well-defined, while for a Jaco-type graph the vertex out-degree $d^+(v)$ is well-defined. Hence, the main challenge in the studies on a pure Jaco graph is to determine $d^-(v)$ and $d^+(v)$ separately where as the main problem in the studies on Jaco-type graphs is to determine $d^-(v)$. 

\section{Jaco-type Graph for the Fibonacci Sequence}

The definition of the infinite Jaco-type graph corresponding to the Fibonacci sequence, which is also called the \textit{Fibonaccian Jaco-type graph}, can be derived from Definition \ref{Defn-1.1}. We have the graph $J_\infty(s_1)$, defined by $V(J_\infty(s_1))=\{v_i: i \in \N\}, \ A(J_\infty(s_1)) \subseteq \{(v_i, v_j): i, j \in \N, i< j\}$ and $(v_i,v_ j) \in A(J_\infty(s_1))$ if and only if $f_i+i\ge j$.

\ni Figure \ref{fig:Fig-1.jpg} depicts $J_{12}(s_1)$. 

\begin{figure}[h!]
\centering
\includegraphics[width=0.7\linewidth]{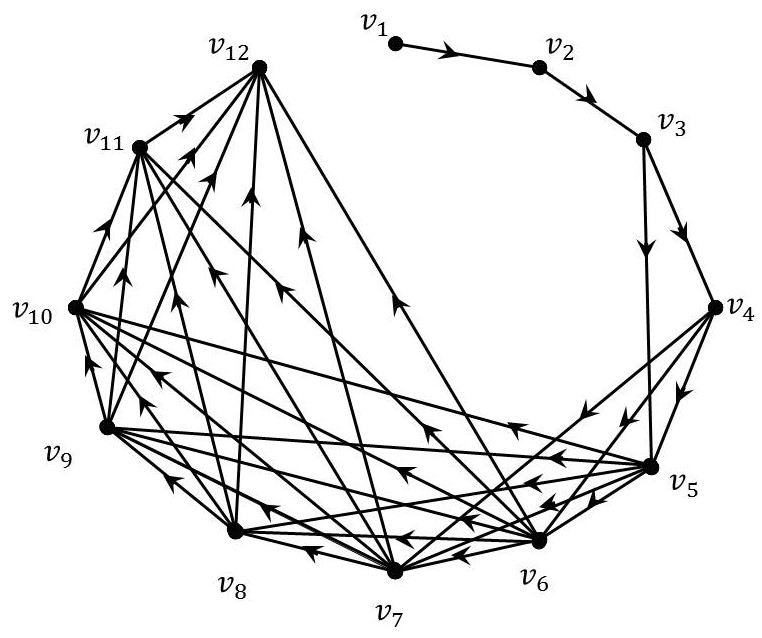}
\caption{$J_{12}(s_1)$.}
\label{fig:Fig-1.jpg}
\end{figure}

Table \ref{Tab-1} depicts the manually calculated invariant, $d^-(v_i)$, $1\le i \le 30$ together with the suggested pattern for $i \ge 6$ which requires proof to settle the determination of the corresponding in-degrees, $d^-(v_i)$, $ i = 3,4,5,\ldots$

\begin{longtable}{|l|l|l|c|l|l|l|}
\hline
$\phi (v_i)\rightarrow i\in{\N}$ & $d^-(v_i)$ & $d^-(v_i)$ & & $\phi (v_i)\rightarrow i\in{\N}$ & $d^-(v_i)$ & $d^-(v_i)$\\
\hline
1 & 0 & - & & 16 & 9 & $f_6+1$\\
\hline
2 & 1 & - & & 17 & 10 & $f_6+2$\\
\hline
3 & 1 & - & & 18 & 11 & $f_6+3$\\
\hline
4 & 1 & - & & 19 & 12 & $f_7-1$\\
\hline
5 & 2 & - & & 20 & 13 & $f_7$\\
\hline
6 & 2 & $f_4-1$ & & 21 & 13 & $f_7$\\
\hline
7 & 3 & $f_4$ & & 22 & 14 & $f_7+1$\\
\hline
8 & 3 & $f_4$ & & 23 & 15 & $f_7+2$\\
\hline
9 & 4 & $f_5-1$ & & 24 & 16 & $f_7+3$\\
\hline
10 & 5 & $f_5$ & & 25 & 17 & $f_7+4$\\
\hline
11 & 5 & $f_5$ & & 26 & 18 & $f_7+5$\\
\hline
12 & 6 & $f_5+1$ & & 27 & 19 & $f_7+6$\\
\hline
13 & 7 & $f_6-1$ & & 28 & 20 & $f_8-1$\\
\hline
14 & 8 & $f_6$ & & 29 & 21 & $f_8$\\
\hline
15 & 8 & $f_6$ & & 30 & 21 & $f_8$\\
\hline
\caption{}\label{Tab-1}
\end{longtable}

We observe that for $i\ge 6$ the subsequences of in-degrees are seemingly of the form: $\{\ldots,f_k-1$, $f_k$, $f_k$, $f_k+1$, $f_k+2$, $f_k+3$,\ldots, $f_k +(f_{k+1}-2),\ldots\}$, $k=4,5,6,\ldots$. 

\vspace{0.2cm}

The following theorem is of importance to prove the aforesaid observation and other results related to both pure Jaco graphs and Jaco-type graphs.

\begin{thm}\label{Thm-2.1}
For any non-negative, stepwise non-decreasing and stepwise increasing integer sequence $\{a_n\}$, and any $\ell \in \N$ there exists at least one vertex $v_i, \ i \in \N$ in the corresponding Jaco-type graph $J_\infty(\{a_n\})$ such that $d^-(v_i) = \ell$.
\end{thm}
\begin{proof}
Consider a non-negative, step-wise non-decreasing and step-wise increasing integer sequence $\{a_n\}$, and assume for some $\ell \in \N, \ d^-(v_i) \ne \ell$, $\forall i \in \N$. Assume without loss of generality that there exists at least one vertex $v_j$ with $d^-(v_j) = \ell-1$, then select $j^* =\max\{j\}$ for which it holds. 

Further assume without loss of generality that $d^-(v_{j^*+1}) = \ell + 1$. Now, for vertex $v_{j^*}$, clearly the lowest subscripted tail vertex of an incident arc is  $v_{j^* -d^-(v_{j^*})}$. With regard to the arcs incident with the vertex $v_{j^* + 1}$, at least all among the arcs  $(v_{j^* -d^-(v_{j^*})},\, v_{j^*+1}),( v_{j^* -d^-(v_{j^*})+1},\,v_{j^*+1}),( v_{j^* -d^-(v_{j^*})+2},\,v_{j^*+1}),\ldots, (v_{j^*},\,v_{j^*+1})$ exist. However, we have $d^-(v_{j^*+1}) = \ell$. Hence, an additional arc, $(v_{j^* -d^-(v_{j^*})-1},v_{j^*+1})$ is required to ensure that $d^-(v_{j^*+1}) = \ell+1$. By Definition \ref{Defn-1.1}, we have a contradiction in that, $d^-(v_{j^*}) = \ell \ne \ell-1$. 

By similar argument leading to contradiction, we can establish that it is not possible to find $d^-(v_{j^*}) = \ell-m$, $d^-(v_{j^*+1}) = \ell +t,\ m,t>1$.

Hence, for all $\ell \in \N$, there exists at least one vertex $v_i, \, i \in \N$ in the corresponding Jaco-type graph $J_\infty(\{a_n\})$ such that $d^-(v_i) = \ell$.
\end{proof}

Before, going to the next theorem, we note some interesting properties of the Fibonacci sequence. Consider the following table of first few elements of the Fibonacci sequence. 

\vspace{0.25cm}
\begin{center}
\begin{tabular}{|c|c|c|c|c|c|c|c|c|c|c|c|c|c|c|c|c|}
\hline
$n$ & 0 & 1 & 2 & 3 & 4 & 5 & 6 & 7 & 8 & 9 & 10 & 11 & 12 & 13 & 14 & 15 \\
\hline
$f_n$ & 0 & 1 & 1 & 2 & 3 & 5 & 8 & 13 & 21 & 34 & 55 & 89 & 144 & 233 & 377 & 610  \\
\hline
\end{tabular}

\vspace{0.35cm}

\begin{tabular}{|c|c|c|c|c|c|c|c|c|c|c|}
\hline
$n$ & 16 & 17 & 18 & 19 & 20 & 21 & 22 & 23 & 24 & 25 \\
\hline
$f_n$ & 987 & 1597 & 2584 & 4181 & 6765 & 10946 & 17711 & 28657 & 46368 & 75025 \\
\hline
\end{tabular}
\end{center}

\vspace{0.25cm}

From the above table, we observe the following properties of the Fibonacci sequence.
\begin{enumerate}\itemsep0mm 
\item[(i)] Look at the number $f_3 = 2$. Every $3$-rd number is a multiple of $2$ ($2, 8, 34, 144,\\ 610, \ldots $),
\item[(ii)]  Look at the number $f_4 = 3$. Every $4$-th number is a multiple of $3$ ($3, 21, 144,\\987, \ldots $), 
\item[(iii)]  Look at the number $f_5 = 5$. Every $5$-th number is a multiple of $5$ ($5, 55, 610, \\ 6765,\ldots $). 
\item[(iv)] Proceeding like this, we can see that every $n$-th number is a multiple of $f_n$.

\item[(v)]  Any Fibonacci number that is a prime number must also have a subscript that is a prime number.
\end{enumerate}

It is to be noted that the converse of (v) is true. That is, it is not true that if a subscript is prime, then so is that Fibonacci number. The first case to show this is the $19$-th position (and $19$ is prime) but $f_{19}=4181$ and $f_{19}$ is not prime because $4181=113 \times 37$.

\vspace{0.3cm}

Invoking the above properties, the following theorem discusses the subsequences of indegrees in an infinite Fibonaccian Jaco-type graph.

\begin{thm}\label{Thm-2.2}
For the infinite Fibonaccian Jaco-type graph $J_\infty(s_1)$, the subsequences of indegrees for vertices $v_i$, are:
\begin{enumerate}\itemsep0mm
\item[(i)] $d^-(v_{3i}) = d^-(v_{3(i-1)}+f_3$, for all $i \ge 3,$ with the initial value $d^-(v_6)= 2= f_3$,
\item[(ii)]  $d^-(v_{4i}) = d^-(v_{4(i-1)}+f_4$, for all $i \ge 4,$ with the initial value $d^-(v_8)= 3= f_4$
\item[(iii)]  $d^-(v_{5i}) = d^-(v_{5(i-1)}+f_5$, for all $i \ge 5,$ and $5$ is the least divisor of $j$, with the initial value $d^-(v_5)= 5= f_5$\\ and 
\item[(iv)]  $d^-(v_j) =  d^-(v_{4m} ) \pm 1$ or $d^-(v_j) =  d^-(v_{4m}) \pm 3$.
\end{enumerate}
\end{thm}
\begin{proof}
For the infinite Fibonaccian Jaco-type graph $J_{\infty}(s_1)$, we would like to determine the  in-degrees for vertices $v_i,\ i \ge 1$. First of all, note that any Jaco-type graph admit a unique linear ordering of the vertices with respect to its definition of the arcs. Let the vertices of the Jaco-type graph be linearly ordered as $v_1,\ 1 \le i \le n$. Label on the vertices of the Jaco-type graph with the numbers from the Fibonacci sequence in the order of which the vertices are linearly ordered. That is, label the vertex $v_i$ with $f_i$, the $i$-th number from the Fibonacci sequence $\{f_1, f_2, f_3 \ldots\}$, where  the $j$-th vertex $v_j$ is labelled as $f_j = f_{j-1}+j_{j-2}$. Let $v_i$ be the minimum subscripted  vertex for $v_j$ such that $i+f_i = j$, where $f_i$ the corresponding labelling of $v_i$. That is, at the $j$-th vertex $i+f_i $ attain the maximum. In this case, clearly $d^-(v_j)= f_i$. Hence, determination of the minimum subscripted vertex say $v_i$ is important.

But, there are vertices $v_j$, where there exists no such minimum subscripted vertex so as to compute the in-degree. Also, note that for all the vertices $\{v_l\}$ between $v_i$ and $v_j$, which do not have such minimum subscripted vertex $v_k$ for which $d^-(v_l)= f_k$, we have $d^-(v_l)= l-k$.

We have every $v_{3i}, i \ge 1$ is a multiple of $2$; every $v_{4i}, i \ge 1 $ is a multiple of $3$; every $v_{5i}, i \ge 1 $ is a multiple of $5$; every $v_{6i}, i \ge 1 $ is a multiple of $8$. 

\textit{Case 1:} Consider the  pairs of vertices $(v_i, v_j)$ such that  $v_{j}, j \ge 3$ is a multiple of $3$. Also note that $f_3= 2$, and $d^-(v_3)=1$. Then, $d^-(v_6)= 2 = f_3$, $d^-(v_9)= 4= d^-(v_6) + 2$, $d^-(v_{12})= 6= d^-(v_6)+2$,$d^-(v_{15})= 8= d^-(v_9)+2$,$\dots$.  In general, $d^-(v_{3i}) = d^-(v_{3(i-1)}+2 (=f_3)$, for all $i \ge 3,$ with the initial value $d^-(v_6)= 2= f_3$. Hence,  $d^-(v_{3i}) = d^-(v_{3(i-1)}+f_3$, for all $i \ge 3,$ with the initial value $d^-(v_6)= 2= f_3$.

\textit{Case 2:} Consider the  pairs of vertices $(v_i, v_j)$ such that  $v_{j}, j \ge 4 $ is a multiple of $4$. Also note that $f_4= 3$, and $d^-(v_4)=1$. Then, $d^-(v_8)= 3 = f_4$, $d^-(v_{12})= 6= d^-(v_4) + 3$, $d^-(v_{16})= 9= d^-(v_{12})+3$,$d^-(v_{20})= 12= d^-(v_{16})+3, \dots$. In general, $d^-(v_{4i}) = d^-(v_{4(i-1)}+f_4$, for all $i \ge 4$ with the initial value $d^-(v_8)=3=f_4$.

\textit{Case 3:} Consider the  pairs of vertices $(v_i, v_j)$ such that  $v_{j}, j \ge 5 $ is a multiple of $5$, (but not divisible by $3$ and $4$) and $5$ is the last divisor of $j$. Also note that $f_5= 5$, and $d^-(v_5)=2$. Then, 

\begin{alignat*}{3}
d^-(v_{10}) & = & 5 & =  f_5,\\
d^-(v_{15}) & = & 8 & =  d^-(v_4) + 3,\\ 
d^-(v_{20}) & = & 12 & =  d^-(v_{12})+3,\\  
d^-(v_{25}) & = & 17 & =  d^-(v_{20})+5,\\
\ldots & &\ldots & \ldots\ldots\dots\\
\ldots & &\ldots & \ldots\ldots\dots.
\end{alignat*}
In general, $d^-(v_{5i}) = d^-(v_{5(i-1)}+f_5$, for all $i \ge 5,$ and $5$ is the least divisor of $j$, with the initial value $d^-(v_5)= 5= f_5$.

\textit{Case 4:} If in the $v_j$-th position we have a prime number. The first such prime number is $7$. and $f_7 = 13$. Also we know that any prime number is of the form $4m \pm 1$ or $4m \pm 3$. 

\textit{Subcase 4.1:} When  $j = 4m \pm 1$. In this case $d^-(v_j) =  d^-(v_{4m} ) \pm 1$ and if  $j = 4m \pm 3$, then case $d^-(v_j) =  d^-(v_{4m} ) \pm 3$.

\ni This completes the proof.  
\end{proof}

\begin{rem}{\rm 
It is interesting to note that for $j \ge 7$ and $j=p_1$, a prime number and the next immediate prime greater than $P_1$ be $p_2$, then  then $d^-(v_{p_2}) = p_2-p_1$.		
}\end{rem}

The generalisation of the Fibonacci numbers is given by the Horadam sequence defined by 

\begin{eqnarray*}
H_0 & = & p\in\N_0,\\
H_1 & = & q\in\N_0,\\
H_n & = & rH_{n-1}+sH_{n-2}.
\end{eqnarray*}  where $r,s \in \N_0$.

\vspace{0.25cm}

In view of the above mentioned generalisation of Fibonacci sequence, we strongly believe that the following conjecture hold.

\begin{conj}\label{Conj-1}
For the Horadam Jaco-type graph, $J_\infty(\{H_n\})$, the in-degree subsequences for vertices $v_i$, for sufficiently large $i$ are of the form $\{\ldots,H_k-1$, $H_k$, $H_k$, $H_k+1$, $H_k+2$, $H_k+3$,\ldots, $H_k +(H_{k+1}-2),\ldots\}$, $k=4,5,6,\ldots$
\end{conj}
 
\subsection{Modular Jaco-type Graph}

It is well known that for the set $\N_0$ of all non-negative integers and $n,k\in \N,\ k\ge 2$, modular arithmetic allows an integer mapping in respect of modulo $k$ as follows. 

\vspace{-0.5cm}

\begin{alignat*}{3}
0 &\mapsto & 0 & =\ m_0\\
1 & \mapsto & 1 & =\ m_1 \\
2 & \mapsto & 2 & =\ m_2 \\
\ldots & & \ldots &  \ldots\\
k-1 & \mapsto & ~~ k-1 & =\ m_{k-1}\\ 
k & \mapsto & 0 & =\  m_{k}\\
k+1 & \mapsto & 1 & =\ m_{k+1}\\
\ldots & & \ldots &  \ldots\\
\end{alignat*}

\vspace{-0.5cm}

The new family of Jaco-type graphs, also called the \textit{modular Jaco-type graphs}, resulting from mod $k, \ k \in \N$ requires a relaxation of Definition \ref{Defn-1.1} to allow a stepwise non-negative, non-decreasing sequence.

Let $s_2 =\{a_n\}$, $ a_n \equiv n($mod $k)= m_n$. Consider the infinite \textit{root}-graph $J_\infty(s_2)$ and define $d^+(v_i) = m_i$, for $i = 1,2,3,\ldots$. From the aforesaid definition it follows that the case $k=1$ will result in a null (edgeless) Jaco-type graph for all $n\in\N$. For $k=2$ and $n$ is even, the Jaco-type graph is the union of $\frac{n}{2}$ copies of directed $P_2$. For $k=3$, the Jaco-type graph is a directed tree and hence is an acyclic graph $G$. 

For illustration, if $k = 5$, then Figure \ref{fig:Fig-2.jpg} depicts $J_{18}(s_2)$.

\begin{figure}[h!]
\centering
\includegraphics[width=0.7\linewidth]{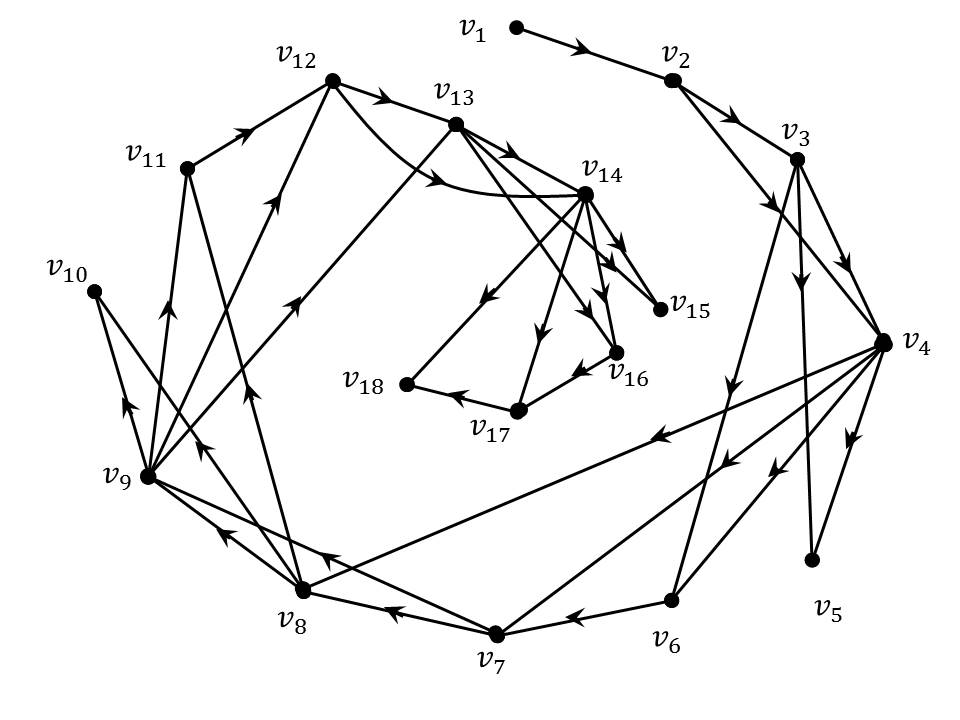}
\caption{$J_{18}(s_2)$. }
\label{fig:Fig-2.jpg}
\end{figure}

Table \ref{Tab-2} depicts the manually calculated invariant, $d^-(v_i)$, $1\le i \le 30$ for $k=8$ together with the suggested pattern for all even $k\ge 2$, $i \ge 1$ which requires proof to settle the determination of the corresponding in-degrees, $d^-(v_i)$, $ i = 1,2,3,\ldots$. 

\begin{longtable}{|l|l|l|c|l|l|l|}
\hline
$\phi (v_i)\rightarrow i\in{\N}$ & $d^-(v_i)$ & $d^-(v_i)=$ & & $\phi (v_i)\rightarrow i\in{\N}$ & $d^-(v_i)$ & $d^-(v_i)=$\\
\hline
1 & 0 & - & & 16 & 4 & $\frac{k}{2}$\\
\hline
2 & 1 & 1 & & 17 & 3 & $\frac{k}{2}-1$\\
\hline
3 & 1 & 1 & & 18 & 4 & $\frac{k}{2}$\\
\hline
4 & 2 & 2 & & 19 & 3 & $\frac{k}{2}-1$\\
\hline
5 & 2 & 2 & & 20 & 4 & $\frac{k}{2}$\\
\hline
6 & 3 & 3 & &  21 & 3 & $\frac{k}{2}-1$\\
\hline
7 & 3 & 3 & & 22 & 4 & $\frac{k}{2}$\\
\hline
8* & 4 & $\frac{k}{2}$ & & 23 & 3 & $\frac{k}{2}-1$\\
\hline
9 & 3 & $\frac{k}{2}-1$ & & 24 & 4 & $\frac{k}{2}$\\
\hline
10 & 4 & $\frac{k}{2}$ & & 25 & 3 & $\frac{k}{2}-1$\\
\hline
11 & 3 & $\frac{k}{2}-1$ & & 26 & 4 & $\frac{k}{2}$\\
\hline
12 & 4 & $\frac{k}{2}$ & & 27 & 3 & $\frac{k}{2}-1$\\
\hline
13 & 3 & $\frac{k}{2}-1$ & & 28 & 4 & $\frac{k}{2}$\\
\hline
14 & 4 & $\frac{k}{2}$ & & 29 & 3 & $\frac{k}{2}-1$\\
\hline
15 & 3 & $\frac{k}{2}-1$ & & 30 & 4 & $\frac{k}{2}$\\
\hline
\caption{$k= 8$}\label{Tab-2}
\end{longtable}

\ni We note that for $i\ge 1$ and $k$ is even, the in-degree sequence seems to have the form $\{\underbrace{0,1,1,2,2,3,3,\ldots,\frac{k}{2}-1,\frac{k}{2}-1,\underbrace{\frac{k}{2}}_{1~entry}}_{first~k~in-degrees},~\underbrace{\frac{k}{2}-1,\frac{k}{2},\frac{k}{2}-1,\frac{k}{2},\ldots, \frac{k}{2}-1,\frac{k}{2}}_{repetitive~subsequence,~k~in-degrees},~\ldots\}$.

Table \ref{Tab-3} depicts the manually calculated invariant, $d^-(v_i),\ 1\le i \le 30$ for $k=9$ together with the suggested pattern for all odd $k\ge 1$, $i \ge 1$ which requires proof to settle the determination of the corresponding in-degrees, $d^-(v_i)$, $ i = 1,2,3,\ldots$.

\begin{longtable}{|l|l|l|c|l|l|l|}
\hline
$\phi (v_i)\rightarrow i\in{\N}$ & $d^-(v_i)$ & $d^-(v_i)=$ & & $\phi (v_i)\rightarrow i\in{\N}$ & $d^-(v_i)$ & $d^-(v_i)=$\\
\hline
1 & 0 & - & & 16 & 4 & $\lfloor\frac{k}{2}\rfloor$\\
\hline
2 & 1 & 1 & & 17 & 4 & $\lfloor\frac{k}{2}\rfloor$\\
\hline
3 & 1 & 1 & & 18 & 4 & $\lfloor\frac{k}{2}\rfloor$\\
\hline
4 & 2 & 2 & & 19 & 4 & $\lfloor\frac{k}{2}\rfloor$\\
\hline
5 & 2 & 2 & & 20 & 4 & $\lfloor\frac{k}{2}\rfloor$\\
\hline
6 & 3 & 3 & & 21 & 4 & $\lfloor\frac{k}{2}\rfloor$\\
\hline
7 & 3 & 3 & & 22 & 4 & $\lfloor\frac{k}{2}\rfloor$\\
\hline
8 & 4 & $\lfloor\frac{k}{2}\rfloor$ & & 23 & 4 & $\lfloor\frac{k}{2}\rfloor$\\
\hline
9* & 4 & $\lfloor\frac{k}{2}\rfloor$ & & 24 & 4 & $\lfloor\frac{k}{2}\rfloor$\\
\hline
10 & 4 & $\lfloor\frac{k}{2}\rfloor$ & & 25 & 4 & $\lfloor\frac{k}{2}\rfloor$\\
\hline
11 & 4 & $\lfloor\frac{k}{2}\rfloor$ & & 26 & 4 & $\lfloor\frac{k}{2}\rfloor$\\
\hline
12 & 4 & $\lfloor\frac{k}{2}\rfloor$ & & 27 & 4 & $\lfloor\frac{k}{2}\rfloor$\\
\hline
13 & 4 & $\lfloor\frac{k}{2}\rfloor$ & & 28 & 4 & $\lfloor\frac{k}{2}\rfloor$\\
\hline
14 & 4 & $\lfloor\frac{k}{2}\rfloor$ & & 29 & 4 & $\lfloor\frac{k}{2}\rfloor$\\
\hline
15 & 4 & $\lfloor\frac{k}{2}\rfloor$ & & 30 & 4 & $\lfloor\frac{k}{2}\rfloor$\\
\hline
\caption{$k= 9$.}\label{Tab-3}
\end{longtable}

We observe that for $i\ge 1$ and $k$ is odd, the sequence of in-degrees seems to have the form $\{\underbrace{0,1,1,2,2,3,3,\ldots,\lfloor\frac{k}{2}\rfloor,\lfloor\frac{k}{2}\rfloor}_{first~k~in-degrees},~\underbrace{\lfloor\frac{k}{2}\rfloor,\lfloor\frac{k}{2}\rfloor,\lfloor\frac{k}{2}\rfloor,\ldots}_{all~in-degrees}\}$.

\begin{thm}\label{Thm-2.3}
Consider the infinite modular Jaco-type graph $J_\infty(s_2)$, modulo $k\ge 1$. If $k$ is even, then the sequence of in-degrees for vertices $v_i,\, i\ge 1$ are of the form 
$\{\underbrace{0,1,1,2,2,3,3,\ldots,\frac{k}{2}-1,\frac{k}{2}-1,\underbrace{\frac{k}{2}}_{1~entry}}_{first~k~in-degrees},~\underbrace{\frac{k}{2}-1,\frac{k}{2},\frac{k}{2}-1,\frac{k}{2},\ldots, \frac{k}{2}-1,\frac{k}{2}}_{repetitive~subsequence,~k~in-degrees},~\ldots\}$ and if $k$ is odd. then the sequence of in-degrees for vertices $v_i,\, i\ge 1$ are of the form
$\{\underbrace{0,1,1,2,2,3,3,\ldots,\lfloor\frac{k}{2}\rfloor,\lfloor\frac{k}{2}\rfloor}_{first~k~in-degrees},~\underbrace{\lfloor\frac{k}{2}\rfloor,\lfloor\frac{k}{2}\rfloor,\lfloor\frac{k}{2}\rfloor,\ldots}_{all~in-degrees}\}$.
\end{thm}
\begin{proof}
Partition $\N$ into subsets $\mathcal{C}_m =\{j: (m-1)k +1 \le j \le mk$, $k \in \N\}$, $m=1,2,3,\ldots$ Also partition the vertex set $V(J_\infty(s_2))$ into subsets $\mathcal{V}_m =\{v_j: j \in \mathcal{C}_m\}$. Clearly, the induced subgraphs $\langle \mathcal{V}_r\rangle$ and $\langle\mathcal{V}_q\rangle$ are isomorphic.

\vspace{0.2cm}

\textit{Case 1:} Let $k\ge 2$ and even. First consider $\langle \mathcal{V}_1\rangle$. For $k=2$ the sequence of in-degrees is $\{0,\frac{k}{2}=\frac{2}{2}=1\} = \{0,1\}$. For $k=4$ the sequence of in-degrees is $\{0, 1,1,\frac{k}{2}=\frac{4}{2}=2\} = \{0,1,1,2\}$. Hence, the result holds for $k=2,4$. 

\vspace{0.2cm}

Assume that it holds for $ k = \ell$, $\ell$ is even. Hence, the corresponding induced subgraph $\langle \mathcal{V}_1\rangle$ has the in-degree sequence $\{0,1,1,2,2,3,3,\ldots,\frac{\ell}{2}-1,\frac{\ell}{2}-1, \underbrace{\frac{\ell}{2}}_{1~entry}\}$. All the out-arcs defined for vertices $v_1,v_2,v_3,\ldots,v_{\frac{k}{2}}$ have heads within $\langle \mathcal{V}_1\rangle$. However, vertices $v_i, \ \frac{\ell}{2}+1 \le i \le \frac{\ell}{2}+ (\frac{\ell}{2}-1)$ requires $2i$ out-arcs in a sufficiently large modular Jaco-type graph. Hence, by adding the required out-arcs by utilising $\langle \mathcal{V}_1\rangle $ and $\langle \mathcal{V}_2\rangle$ to construct $J_{2\ell}(s_2)$, the corresponding sequence of in-degrees is, $\{0,1,1,2,2,3,3,\ldots,\frac{\ell}{2}-1,\frac{\ell}{2}-1, \underbrace{\frac{\ell}{2}}_{1~entry},\underbrace{\frac{\ell}{2}-1, \frac{\ell}{2},\frac{\ell}{2}-1, \frac{\ell}{2},\ldots,\frac{\ell}{2}-1, \frac{\ell}{2}}_{\ell~in-degrees}\}$. Since the in-degree of any vertex $v_i$ in Jaco-type graph of any finite size or infinite, remains constant, the result follows for the in-degree of verices $v_{\ell+1}$, $v_{\ell+2}$, $v_{\ell+3},\ldots$. Hence, the result holds for $J_\infty(s_2)$, and $k=\ell$. Since the same reasoning applies for $k=\ell+2$ mathematical induction immediately implies that the general result follows for $J_\infty(s_2)$, $\forall$ even $k\in \N$.\\

\textit{Case 2:} Let $k\ge 1$ and odd. The proof follows through similar reasoning to that of Case 1.
\end{proof}

Note that the technique used in the proof of Theorem \ref{Thm-2.3} is called \textit{looped mathematical induction}. 

For a given $k$ the in-degree for a vertex $v_i$ in both $J_\infty(s_2)$ and the finite $J_n(s_2)$ remains equal and hence the next corollary is immediate consequence of Theorem \ref{Thm-2.3}.

\begin{cor}\label{Cor-2.4}
For a modular Jaco-type graph, mod $k\ge 1$ we have
\begin{enumerate}\itemsep0mm
	\item[(i)] If $k$ is even and $i\ge k$ then $d^-(v_i) =
	\begin{cases}
	\frac{k}{2}-1;  & i = 1 (\rm{mod}\ k),\\
	\frac{k}{2}; & \text{Otherwise}. 
	\end{cases}$
	\item[(ii)] If $k$ is odd and $i\ge k-1$ then, $d^-(v_i) = \lfloor\frac{k}{2}\rfloor$.
\end{enumerate}
\end{cor}

\vspace{0.2cm}

In the study of Jaco-type graphs, the concepts of the prime Jaconian vertex denoted, $v_p$ and the Jaconian set are of importance. For ease of reference the adapted definitions from \cite{KSS1} are repeated here.

\begin{defn}{\rm \cite{KSS1}
The set of vertices attaining degree $\Delta (J_n(s_2))$ is called the set of Jaconian vertices; the Jaconian vertices or the Jaconian set of the Jaco-type graph $J_n(s_2)$, and denoted, $\J(J_n(s_2))$ or, $\{J_n(s_2)\}$ for brevity.
}\end{defn}

\begin{defn}{\rm \cite{KSS1}
The lowest numbered (subscripted) Jaconian vertex is called the prime Jaconian vertex of a Jaco-type graph and denoted, $v_p$.
}\end{defn}

For $k\ge 3$, the modular Jaco-type graph is connected. For connected modular Jaco-type graphs we have the next result.

\begin{prop}\label{Prop-2.5}
For the infinite modular Jaco-type graph $J_\infty(s_2),\, k\ge 3$, we have 
\begin{equation*} 
\J(J_\infty(s_2)) =
\begin{cases}
\{v_{k-1},v_{2k-2},v_{2k-1},v_{3k-2},v_{3k-1},\ldots\},  & \text {if $k$ even,}\\
\{v_{k-1},v_{2k-1},v_{3k-1},\ldots\},  & \text {if $k$ odd.}
\end{cases}
\end{equation*}
\end{prop}
\begin{proof}
Note that $\Delta(J_\infty(s_2))$ is the maximum degree attained by some vertices. Hence, $\Delta(J_\infty(s_2)) = \max\{d^+(v_i) + d^-(v_i)\}$ over all $i\in\N$. Since the $\max\{\ell\}$ (mod $k$) is defined for $\ell=k-1$, the maximum out-degrees are obtained for  vertices subscripted with $t\cdot k-1,\ t=1,2,3,\ldots$. The aforesaid implies that the results for both $k$ even or $k$ odd follow directly from Theorem \ref{Thm-2.3}.
\end{proof}

\section{Conclusion}

Jaco-type graphs present a wide scope for research in respect of the many known invariants applicable to graphs. It is noted that all Jaco-type graphs defined for non-negative, step-wise non-decreasing and step-wise increasing integer sequences $\{a_n\}$, are propagating graphs \cite{KSCM1}. Hence, a wide scope for further research exists with regards to black clouds, black arcs and black energy dissipation.

It was reported that the On-line Encyclopedia of Integer Sequences (OEIS) hosts about 2.6 lakhs of sequences. Amongst the sequences, it is likely that thousands of integer sequences exist for which Jaco-type graphs can be defined. Characterising the Horadam Jaco-type graph is also an open research topic.

\end{document}